\documentclass{article}
\usepackage[utf8]{inputenc}
\usepackage[margin=1in]{geometry}
\usepackage{graphicx}
\usepackage{amssymb}
\usepackage{amsthm}
\usepackage{epstopdf}
\usepackage{dsfont}
\usepackage{xcolor}
\usepackage{amsmath}
\usepackage{mathrsfs}
\usepackage[parfill]{parskip}
\usepackage{comment}
\usepackage{tikz}
\usepackage{tikz-3dplot}
\usetikzlibrary{intersections}
\usepackage{mathtools}
\usepackage{hyperref}
\usepackage{caption}
\newtheorem{theorem}{Theorem}[section]
\newtheorem{proposition}[theorem]{Proposition}
\newtheorem{lemma}[theorem]{Lemma}
\newtheorem{corollary}[theorem]{Corollary}
\theoremstyle{definition}
\newtheorem{definition}[theorem]{Definition}

\begin{document}

\title{A Positive Answer to B\'ar\'any's Question on Face Numbers of Polytopes}

\author{Joshua Hinman\\
\small Department of Mathematics\\
\small University of Washington\\
\small Seattle, WA 98195-4350, USA\\
\small \texttt{joshrh@uw.edu}}
\date{}

\maketitle

\begin{abstract}
    Despite a full characterization of the face vectors of simple and simplicial polytopes, the face numbers of general polytopes are poorly understood. Around 1997, B\'ar\'any asked whether for all convex $d$-polytopes $P$ and all $0 \leq k \leq d-1$, $f_k(P) \geq \min\{f_0(P), f_{d-1}(P)\}$. We answer B\'ar\'any's question in the affirmative and prove a stronger statement: for all convex $d$-polytopes $P$ and all $0 \leq k \leq d-1$,
    \[
    \frac{f_k(P)}{f_0(P)} \geq \frac{1}{2}\biggl[{\lceil \frac{d}{2} \rceil \choose k} + {\lfloor \frac{d}{2} \rfloor \choose k}\biggr],
    \qquad
    \frac{f_k(P)}{f_{d-1}(P)} \geq \frac{1}{2}\biggl[{\lceil \frac{d}{2} \rceil \choose d-k-1} + {\lfloor \frac{d}{2} \rfloor \choose d-k-1}\biggr].
    \]
    In the former, equality holds precisely when $k=0$ or when $k=1$ and $P$ is simple. In the latter, equality holds precisely when $k=d-1$ or when $k=d-2$ and $P$ is simplicial.
\end{abstract}

\section{Introduction}

This paper investigates the face numbers of polytopes. If $P$ is a $d$-polytope and $0 \leq k \leq d-1$, we define the \emph{$k^\text{th}$ face number} $f_k(P)$ as the number of $k$-dimensional faces of $P$, and the \emph{face vector} or \emph{$f$-vector} as $(f_0(P), \ldots, f_{d-1}(P))$. While easily defined and extensively studied, face numbers still remain highly mysterious.

Face numbers are best understood in the case of simple and simplicial polytopes. In 1971, McMullen published the famous \emph{$g$-conjecture}, a proposed set of necessary and sufficient conditions for a vector $(f_0, \ldots, f_{d-1})$ to be the $f$-vector of some simplicial polytope \cite{mcmullen71}. Billera and Lee proved the sufficiency of these conditions in 1980 \cite{billera80}\cite{billera81}; Stanley proved their necessity the same year \cite{stanley80}. The resulting \emph{$g$-theorem} gives a full characterization of the $f$-vectors of simplicial, and consequently simple, polytopes.

Results on general polytopes are far more elusive. McMullen proved in 1970 that for a fixed dimension $d$ and number of vertices $v > d$, the \emph{cyclic polytope} $C(v, d)$ simultaneously maximizes all face numbers \cite{mcmullen70}. Some is also known about \emph{flag $f$-numbers}, which count the chains of faces $\varnothing \subset G_1 \subset \cdots \subset G_s \subset P$ in a fixed sequence of dimensions. Significant results include Kalai's rigidity inequality for flag $f$-vectors \cite[Theorem 1.4]{kalai87}, as well as Bayer's inequalities for flag $f$-vectors of 4-polytopes \cite{bayer87}.

The greatest hole in our understanding of face numbers is the question of lower bounds on $f_k(P)$. It is well known that for a fixed dimension $d$, the $d$-simplex $\Delta^d$ simultaneously minimizes all face numbers. The most general lower bound result to date is a recent theorem of Xue \cite{xue20}, originally conjectured by Gr\"unbaum in 1967 \cite{grunbaum03}: if $P$ is a $d$-polytope with $f_0(P) = d + s \leq 2d$, then
\[
f_k(P) \geq {d+1 \choose k+1} + {d \choose k+1} - {d+1-s \choose k+1}.
\]

What if we do not know that $f_0(P) \leq 2d$? B\'ar\'any asked (see \cite[(3.4)]{barany98}\cite[Problem 15.6.5]{billera97}) whether for all $d$-polytopes $P$ and all $0 \leq k \leq d-1$, $f_k(P) \geq \min\{f_0(P), f_{d-1}(P)\}$. We aim to resolve B\'ar\'any's question with our main result, a general lower bound on $f_k(P)$ given $f_0(P)$ or $f_{d-1}(P)$. For all integers $0 \leq k < d$, we define
\[
\rho(d, k) = \frac{1}{2}\biggl[{\lceil \frac{d}{2} \rceil \choose k} + {\lfloor \frac{d}{2} \rfloor \choose k}\biggr].
\]
Our result is as follows: for all $d$-polytopes $P$ and $0 \leq k \leq d-1$,
\begin{align}
\frac{f_k(P)}{f_0(P)} &\geq \rho(d, k), \label{eqn:bound}\\
\frac{f_k(P)}{f_{d-1}(P)} &\geq \rho(d, d-k-1). \label{eqn:dual}
\end{align}
Note that (\ref{eqn:bound}) is nontrivial when $k \leq \lceil \frac{d}{2} \rceil$, while (\ref{eqn:dual}) is nontrivial when $k \geq \lfloor \frac{d}{2} \rfloor - 1$.

Our result answers B\'ar\'any's question and does better. We show that if $P$ is a $d$-polytope and $0 \leq k \leq \lfloor \frac{d}{2} \rfloor$, then $f_k(P) \geq f_0(P)$; likewise, if $\lceil \frac{d}{2} \rceil - 1 \leq k \leq d-1$, then $f_k(P) \geq f_{d-1}(P)$. Notably, if $d$ is odd and $k = \frac{d-1}{2}$, then $\frac{f_k(P)}{f_0(P)}, \frac{f_k(P)}{f_{d-1}(P)} \geq \rho(d, k) = \frac{k}{2} + 1$.

Bounds (\ref{eqn:bound}) and (\ref{eqn:dual}) are strongest when $k \approx \frac{d}{4}$ and $k \approx \frac{3d}{4}$, respectively. These values evoke Bj\"orner's results on the partial unimodality of $f$-vectors \cite{bjorner94}. Bj\"orner proved that for simplicial $d$-polytopes, $f_0 < \cdots < f_{\lfloor \frac{d}{2} \rfloor - 1} \leq f_{\lfloor \frac{d}{2} \rfloor}$ and $f_{\lfloor \frac{3(d-1)}{4} \rfloor} > \cdots > f_{d-1}$. Dually, for simple $d$-polytopes, $f_0 < \cdots < f_{\lceil \frac{d-1}{4} \rceil}$ and $f_{\lceil \frac{d}{2} \rceil - 1} \geq f_{\lceil \frac{d}{2} \rceil} > \cdots > f_{d-1}$. In fact, Bj\"orner proved that these are the strongest unimodality results possible for simplicial and simple polytopes.

Our proof relies on \emph{solid angles}, a generalization of angles in arbitrary finite dimension. Our primary information on solid angles is due to Perles and Shephard \cite{perles67}\cite{shephard68}. For us, the key feature distinguishing the boundary of a $d$-polytope from an arbitrary $(d-1)$-complex is that of angle deficiencies where several facets meet. Two results of Perles and Shephard (see Theorems \ref{perles}-\ref{curvature} below) allow us to use this notion in proving (\ref{eqn:bound}) and (\ref{eqn:dual}).

The structure of our paper is as follows. Section \ref{s:preliminaries} provides the background for our work. We include discussion of polytopes (\S\ref{subs:polytopes}) and solid angles (\S\ref{subs:angles}), as well as a technical lemma (\S\ref{subs:combinatorics}). Section \ref{s:result} is dedicated to our main result. We first prove a key proposition on angle sums (Proposition \ref{angles}). Next, we prove (\ref{eqn:bound}) and (\ref{eqn:dual}) and characterize the cases where equality holds (Theorem \ref{bound}). Finally, we answer B\'ar\'any's question (Corollary \ref{barany}).

%%%%%%%%%%%%%%%%%%%%%%%%%%%%%%%%%%%%%%%%%%%%%%%%%%%%%%%%%%%%%%%%%%%%%%%%%%%%%%%%
%%%%%%%%%%%%%%%%%%%%%%%%%%%%%%%%%%%%%%%%%%%%%%%%%%%%%%%%%%%%%%%%%%%%%%%%%%%%%%%%

\section{Preliminaries}
\label{s:preliminaries}

In this section, we introduce the concepts and first results needed for our proof of (\ref{eqn:bound}) and (\ref{eqn:dual}). Most of these preliminaries revolve around polytopes, polytopal complexes, and solid angles. We also include a lemma about binomial coefficients.

\subsection{Polytopes and Polytopal Complexes}
\label{subs:polytopes}

We begin by introducing some key notions about polytopes and polytopal complexes. All polytopes are assumed to be convex. For any undefined terminology, we refer the reader to \cite{grunbaum03}\cite{ziegler95}.

\begin{definition}
For nonnegative integers $d$, a \emph{$d$-polytope} $P$ is the $d$-dimensional convex hull of a finite point set in some real vector space. By default, we regard $P$ as a subset of $\mathds{R}^d$. The empty set is considered a $(-1)$-polytope.

A \emph{face} of $P$ is either $P$ itself or a polytope $G = H \cap P$, where $H$ is some codimension-1 hyperplane that does not intersect the interior of $P$. A $0$- or $(d-1)$-dimensional face of $P$ is called a \emph{vertex} or \emph{facet}, respectively. For $-1 \leq k \leq d$, we denote by $f_k(P)$ the number of $k$-dimensional faces of $P$.
\end{definition}

\begin{definition}
Let $P$ be a $d$-polytope with vertex set $V$, and let $v \in \mathds{R}^d$ be a vector. We say $v$ is in \emph{general position with respect to $P$} if for all subsets $U \subset V$ such that $\dim \operatorname{aff}(U) < d$, $v$ is not parallel to $\operatorname{aff}(U)$.
\end{definition}

\begin{definition}
Let $P$ be a $d$-polytope and $G$ a $k$-dimensional face of $P$. The \emph{quotient polytope} $P/G$ is defined as $H \cap P$, where $H$ is a $(d-k-1)$-dimensional hyperplane intersecting exactly those faces of $P$ which properly contain $G$.
\end{definition}

\begin{definition}
A \emph{polytopal complex} $\mathcal{C}$ is a finite collection of polytopes with the following properties:
\begin{itemize}
    \item $\varnothing \in \mathcal{C}$.
    \item If $F \in \mathcal{C}$ and $G$ is a face of $F$, then $G \in \mathcal{C}$.
    \item If $F, G \in \mathcal{C}$, then $F \cap G$ is a face of both $F$ and $G$.
\end{itemize}
If $S \subset \mathcal{C}$, we denote with $\langle S \rangle$ the subcomplex of $\mathcal{C}$ generated by $S$. Additionally, if $P$ is a polytope, we denote by $\partial P$ the \emph{boundary complex} consisting of all proper faces of $P$.
\end{definition}

\begin{definition}
Let $P$ be a $d$-polytope. A \emph{polytopal subdivision} of $P$ is a polytopal complex $\mathcal{C}$ whose underlying space is $P$. A $d$-polytope belonging to $\mathcal{C}$ is called a \emph{cell}.
\end{definition}

\begin{definition}
For nonnegative integers $d$, a \emph{$d$-diagram} is a polytopal subdivision $\mathcal{C}$ of some $d$-polytope $P$ such that for all $G \in \mathcal{C}$, $G \cap \partial P$ is a face of $P$. An \emph{interior vertex} of $\mathcal{C}$ is a vertex $x \in \mathcal{C}$ which is not a vertex of $P$.
\end{definition}

We now give an elementary, but useful, result about $d$-diagrams.

\begin{lemma}
\label{d-diagrams}
Any $d$-diagram with $d \geq 1$ contains an interior vertex.
\end{lemma}

\begin{proof}
Let $\mathcal{C}$ be a $d$-diagram with $d \geq 1$, and let $P$ be the $d$-polytope of which $C$ is a subdivision. Let $G$ be a cell of $\mathcal{C}$, so $G \cap \partial P$ is a face of $P$. Note that $G \neq P$, because $P \cap \partial P = \partial P$ is not a face of $P$.

Since $\dim G = \dim P = d$, we know $G$ is not a face of $P$, so $G \cap \partial P \neq G$. It follows that $G$ contains some vertex $x \notin G \cap \partial P$. By definition, $x$ is an interior vertex of $\mathcal{C}$.
\end{proof}

\subsection{Solid Angles}
\label{subs:angles}

Next we discuss solid angles, a higher-dimensional analogue for the familiar angles of polygons. Solid angles are a crucial ingredient in proving our main result.

\begin{definition}
Let $P$ be a $d$-polytope and $G$ a nonempty face of $P$. We define the \emph{solid angle of $P$ at $G$}, denoted $\varphi(P, G)$, as follows. Let $B$ be an open ball centered on $G$, sufficiently small as to intersect only those faces of $P$ which contain $G$. Then
\[
\varphi(P, G) = \frac{m(B \cap P)}{m(B)},
\]
where $m$ is the volume function. For $0 \leq k \leq d$, we denote by $\varphi_k(P)$ the sum of solid angles of $P$ at all $k$-dimensional faces. That is, if $\mathcal{G}_k$ is the set of all $k$-dimensional faces of $P$, then
\[
\varphi_k(P) = \sum_{G \in \mathcal{G}_k} \varphi(P, G).
\]
\end{definition}

We give two essential results on solid angles by Perles and Shephard. The first places a lower bound on the $k^\text{th}$ angle sum using orthogonal projections. The second formalizes our intuition about the positive curvature of polytopes.

\begin{theorem}[{Perles and Shephard \cite[(23)]{perles67}}]
\label{perles}
Let $P$ be a $d$-polytope and $0 \leq k \leq d-1$. For all vectors $v \in \mathds{R}^d$ in general position with respect to $P$, let $H_v \subset \mathds{R}^d$ be a codimension-1 hyperplane orthogonal to $v$, and let $\pi_v:\mathds{R}^d \to H_v$ be the orthogonal projection map. Then
\[
\varphi_k(P) \geq \frac{1}{2}\Bigl[f_k(P) - \max_v f_k(\pi_v(P))\Bigr].
\]
\end{theorem}

\begin{theorem}[{Shephard \cite[Theorem 3]{shephard68}}]
\label{curvature}
Let $P$ be a $d$-polytope and $G$ a $k$-dimensional face of $P$, where $0 \leq k \leq d-2$. Let $\mathcal{F}_G$ be the set of facets of $P$ containing $G$. Then
\[
\sum_{F \in \mathcal{F}_G} \varphi(F, G) \leq 1,
\]
with equality if and only if $k = d-2$.
\end{theorem}

\subsection{A Combinatorial Lemma}
\label{subs:combinatorics}

Finally, we prove a lemma about the convexity of binomial coefficients.

\begin{lemma}
\label{convex}
For all nonnegative integers $a, b, c$,
\[
{a \choose c} + {b \choose c} \geq {\lceil \frac{a+b}{2} \rceil \choose c} + {\lfloor \frac{a+b}{2} \rfloor \choose c}.
\]
\end{lemma}

\begin{proof}
Without loss of generality, suppose $a \geq b$. If $a - b > 1$,
\[
\biggl[{a \choose c} + {b \choose c}\biggr] - \biggl[{a-1 \choose c} + {b+1 \choose c}\biggr] = {a-1 \choose c-1} - {b \choose c-1} \geq 0.
\]
Thus,
\[
{a \choose c} + {b \choose c} \geq {a-1 \choose c} + {b+1 \choose c} \geq {a-2 \choose c} + {b+2 \choose c} \geq \cdots \geq {\lceil \frac{a+b}{2} \rceil \choose c} + {\lfloor \frac{a+b}{2} \rfloor \choose c}.\qedhere
\]
\end{proof}

%%%%%%%%%%%%%%%%%%%%%%%%%%%%%%%%%%%%%%%%%%%%%%%%%%%%%%%%%%%%%%%%%%%%%%%%%%%%%%%%
%%%%%%%%%%%%%%%%%%%%%%%%%%%%%%%%%%%%%%%%%%%%%%%%%%%%%%%%%%%%%%%%%%%%%%%%%%%%%%%%

\section{Main Result}
\label{s:result}

In this section, we prove our main theorem on face numbers and consequently answer B\'ar\'any's question. The key to our proof is the following proposition about angle sums. Recall that for all integers $0 \leq k < d$, we define
\[
\rho(d, k) = \frac{1}{2}\biggl[{\lceil \frac{d}{2} \rceil \choose k} + {\lfloor \frac{d}{2} \rfloor \choose k}\biggr].
\]

\begin{proposition}
\label{angles}
For all $(d-1)$-polytopes $Q$ and all $0 \leq k \leq d-2$,
\[
\varphi_k(Q) \geq \rho(d, d-k-1).
\]
\end{proposition}

\begin{proof}
Let $Q$ be a $(d-1)$-polytope and $0 \leq k \leq d-2$. If $d=2$, then $k=0$, so $\varphi_k(Q) = \rho(d, d-k-1) = 1$. For the remainder of this proof, we will assume $d \geq 3$.

For each vector $v \in \mathds{R}^{d-1}$ in general position with respect to $Q$, let $H_v$ be a codimension-1 hyperplane orthogonal to $v$. Define $\pi_v:\mathds{R}^{d-1} \to H_v$ to be the orthogonal projection map.

Let $\mathcal{F}$ be the set of facets of $Q$. For each $F \in \mathcal{F}$, let $u_F$ be the outer unit normal of $F$. Define subcomplexes $Q_v^+, Q_v^- \subset \partial Q$ as follows:
\begin{align*}
    Q_v^+ &= \langle F \in \mathcal{F} \mid v \cdot u_F > 0 \rangle,\\
    Q_v^- &= \langle F \in \mathcal{F} \mid v \cdot u_F < 0 \rangle.
\end{align*}
Then $Q_v^+ \cup Q_v^- = \partial Q$. Furthermore, $\pi_v|_{Q_v^+}$ and $\pi_v|_{Q_v^-}$ are homeomorphisms onto $\pi_v(Q)$. Since $\partial Q_v^+ = \partial Q_v^- = Q_v^+ \cap Q_v^-$, it follows that $\pi_v|_{Q_v^+ \cap Q_v^-}$ is a homeomorphism onto $\partial \pi_v(Q)$.

Both $\{\pi_v(G) \mid G \in Q_v^+\}$ and $\{\pi_v(G) \mid G \in Q_v^-\}$ are polytopal subdivisions of $\pi_v(Q)$ (see \cite[Definition 5.3]{ziegler95}). Let $\mathcal{Q}_v$ be the smallest common refinement of these two subdivisions:
\[
\mathcal{Q}_v = \{\pi_v(G^+) \cap \pi_v(G^-) \mid G^+ \in Q_v^+, G^- \in Q_v^-\}.
\]
We claim that $\mathcal{Q}_v$ is a $(d-2)$-diagram.

First, we will show that $\mathcal{Q}_v$ is a polytopal subdivision of $\pi_v(Q)$. It is clear that each element of $\mathcal{Q}_v$ is a polytope, and that $\varnothing = \pi_v(\varnothing) \in \mathcal{Q}_v$. Let $G = \pi_v(G^+) \cap \pi_v(G^-) \in \mathcal{Q}_v$, where $G^+ \in Q_v^+$ and $G^- \in Q_v^-$. Then for each face $I$ of $G$, $I = \pi_v(I^+) \cap \pi_v(I^-)$ for some faces $I^+$ of $G^+$ and $I^-$ of $G^-$. Thus, all faces of $G$ belong to $\mathcal{Q}_v$.

Now, let $G_1 = \pi_v(G_1^+) \cap \pi_v(G_1^-) \in \mathcal{Q}_v$ and $G_2 = \pi_v(G_2^+) \cap \pi_v(G_2^-) \in \mathcal{Q}_v$, where $G_1^+, G_2^+ \in Q_v^+$ and $G_1^-, G_2^- \in Q_v^-$. Since $\pi_v$ maps both $Q_v^+$ and $Q_v^-$ homeomorphically onto $\pi_v(Q)$, we know $\pi_v(G_1^+ \cap G_2^+) = \pi_v(G_1^+) \cap \pi_v(G_2^+)$ and $\pi_v(G_1^- \cap G_2^-) = \pi_v(G_1^-) \cap \pi_v(G_2^-)$. It follows that
\[
G_1 \cap G_2 = \pi_v(G_1^+) \cap \pi_v(G_1^-) \cap \pi_v(G_2^+) \cap \pi_v(G_2^-) = \pi_v(G_1^+ \cap G_2^+) \cap \pi_v(G_1^- \cap G_2^-).
\]
We know $G_1^+ \cap G_2^+$ is a face of $G_1^+$, so $\pi_v(G_1^+ \cap G_2^+)$ is a face of $\pi_v(G_1^+)$. Likewise, $G_1^- \cap G_2^-$ is a face of $G_1^-$, so $\pi_v(G_1^- \cap G_2^-)$ is a face of $\pi_v(G_1^-)$. It follows that $G_1 \cap G_2$ is a face of $\pi_v(G_1^+) \cap \pi_v(G_1^-) = G_1$. A similar argument shows that $G_1 \cap G_2$ is a face of $G_2$. We may therefore conclude that $\mathcal{Q}_v$ is a polytopal subdivision of $\pi_v(Q)$.

Suppose as before that $G = \pi_v(G^+) \cap \pi_v(G^-) \in \mathcal{Q}_v$, where $G^+ \in Q_v^+$ and $G^- \in Q_v^-$. Since $\pi_v$ maps $Q_v^+, Q_v^-$ homeomorphically onto $\pi_v(Q)$ and maps $Q_v^+ \cap Q_v^-$ homeomorphically onto $\partial \pi_v(Q)$, we can observe that
\begin{align*}
    \pi_v(G^+ \cap G^-) &= \pi_v((G^+ \cap Q_v^-) \cap (G^- \cap Q_v^+)) = \pi_v(G^+ \cap Q_v^-) \cap \pi_v(G^- \cap Q_v^+),\\
    \pi_v(G^+ \cap Q_v^-) &= \pi_v(G^+ \cap (Q_v^+ \cap Q_v^-)) = \pi_v(G^+) \cap \partial \pi_v(Q),\\
    \pi_v(G^- \cap Q_v^+) &= \pi_v(G^- \cap (Q_v^+ \cap Q_v^-)) = \pi_v(G^-) \cap \partial \pi_v(Q).
\end{align*}
Thus,
\[
G \cap \partial \pi_v(Q) = \pi_v(G^+) \cap \pi_v(G^-) \cap \partial \pi_v(Q) = \pi_v(G^+ \cap Q_v^-) \cap \pi_v(G^- \cap Q_v^+) = \pi_v(G^+ \cap G^-).
\]
Since $G^+ \cap G^- \in Q_v^+ \cap Q_v^-$ is a face of $Q$, it follows that $G \cap \partial \pi_v(Q) = \pi_v(G^+ \cap G^-)$ is a face of $\pi_v(Q)$. This verifies our claim that $\mathcal{Q}_v$ is a $(d-2)$-diagram.

By Lemma \ref{d-diagrams}, $\mathcal{Q}_v$ contains an interior vertex $x$. Let $x = \pi_v(X^+) \cap \pi_v(X^-)$, where $X^+ \in Q_v^+$ and $X^- \in Q_v^-$ (Figure \ref{fig1}). We know $x \notin \partial \pi_v(Q)$, so $X^+, X^- \notin Q_v^+ \cap Q_v^-$. It follows that for any face $G$ of $Q$, if $X^+ \subset G$, then $G$ is a face of $Q_v^+$ but not a face of $Q_v^-$. Similarly, if $X^- \subset G$, then $G$ is a face of $Q_v^-$ but not a face of $Q_v^+$.

\begin{figure}
\centering
\tdplotsetmaincoords{65}{75}
\begin{tikzpicture}[line cap=round, line join=round, tdplot_main_coords, scale=3, line width = 1.25pt, label distance=-4pt]
%\tdplotsetmaincoords{70}{110}
\def\x{3};
\def\y{2};
\def\z{2};
\pgfmathsetmacro\m{\x^2+\y^2+\z^2}
\def\d{2.5};

\coordinate (000) at (\d, 0, 0);
\pgfmathsetmacro\dotooo{(\d*\x)/\m};
\coordinate (000v) at ({\d - \dotooo*\x}, {-\dotooo*\y}, {-\dotooo*\z});
\coordinate (001) at (\d, 0, 1);
\pgfmathsetmacro\dotool{(\d*\x + \z)/\m};
\coordinate (001v) at ({\d - \dotool*\x}, {-\dotool*\y}, {1-\dotool*\z});
\coordinate (010) at (\d, 1, 0);
\pgfmathsetmacro\dotolo{(\d*\x + \y)/\m};
\coordinate (010v) at ({\d - \dotolo*\x}, {1-\dotolo*\y}, {-\dotolo*\z});
\coordinate (011) at (\d, 1, 1);
\pgfmathsetmacro\dotoll{(\d*\x + \y + \z)/\m};
\coordinate (011v) at ({\d - \dotoll*\x}, {1-\dotoll*\y}, {1-\dotoll*\z});
\coordinate (100) at (\d+1, 0, 0);
\pgfmathsetmacro\dotloo{(\d*\x + \x)/\m};
\coordinate (100v) at ({\d+1 - \dotloo*\x}, {-\dotloo*\y}, {-\dotloo*\z});
\coordinate (101) at (\d+1, 0, 1);
\pgfmathsetmacro\dotlol{(\d*\x + \x + \z)/\m};
\coordinate (101v) at ({\d+1 - \dotlol*\x}, {-\dotlol*\y}, {1-\dotlol*\z});
\coordinate (110) at (\d+1, 1, 0);
\pgfmathsetmacro\dotllo{(\d*\x + \x + \y)/\m};
\coordinate (110v) at ({\d+1 - \dotllo*\x}, {1-\dotllo*\y}, {-\dotllo*\z});
\coordinate (111) at (\d+1, 1, 1);
\pgfmathsetmacro\dotlll{(\d*\x + \x + \y + \z)/\m};
\coordinate (111v) at ({\d+1 - \dotlll*\x}, {1-\dotlll*\y}, {1-\dotlll*\z});

\path[fill=blue, opacity = 0.6] (011) -- (010) -- (000) -- (001) -- cycle;
\path[fill=red, opacity = 0.6] (111) -- (011) -- (010) -- (110) -- cycle;
\path[fill=blue, opacity = 0.6] (110) -- (100) -- (000) -- (010) -- cycle;
\draw[dashed] (010) -- (110);
\draw[dashed] (000) -- (010);
\draw[dashed] (010) -- (011);
\path[fill=red, opacity = 0.6] (111) -- (110) -- (100) -- (101) -- cycle;
\path[fill=blue, opacity = 0.6] (101) -- (001) -- (000) -- (100) -- cycle;
\path[fill=red, opacity = 0.6] (111) -- (101) -- (001) -- (011) -- cycle;
\draw (111) -- (110) -- (100) -- (101) -- cycle;
\draw (101) -- (001) -- (000) -- (100);
\draw (001) -- (011) -- (111);
\draw[line width=3pt, shorten <=1pt, shorten >=1pt] (000) -- (001) node [midway,label=left:\large $X^+$]{};
\draw[line width=3pt, shorten <=1pt, shorten >=1pt] (101) -- (111) node [midway, label=below:\large $X^-$]{};
\draw (111v) -- (110v) -- (100v) -- (101v) -- cycle;
\draw (111v) -- (011v) -- (010v) -- (110v);
\draw (101v) -- (001v) -- (011v);
\draw (000v) -- (100v);
\draw (000v) -- (010v);
\draw (000v) -- (001v);
\path[name path=patha] (000v) -- (001v);
\path[name path=pathb] (101v) -- (111v);
\path [name intersections={of=patha and pathb,by=inter}];
\node [circle,fill=black,inner sep=2.5pt, label=below left:\large $x$] at (inter) {};
%\coordinate (vstart) at ({0.5 + \d - 2*\x/5}, {0.5 - 2*\y/5}, {0.5 - 2*\z/5});
%\coordinate (vmid) at ({0.5 + \d - 1*\x/2}, {0.5 - 1*\y/2}, {0.5 - 1*\z/2});
%\coordinate (vend) at ({0.5 + \d - 3*\x/5}, {0.5 - 3*\y/5}, {0.5 - 3*\z/5});
%\draw[->, line width=2.5pt, color=gray] (vstart) -- (vend);
%\node[label=below:\large \textcolor{gray}{$v$}] at (vmid){};
\node[right=1pt,label=right:\huge \textcolor{red}{$Q_v^-$}] at (111){};
\node[left=1pt,label=left:\huge \textcolor{blue}{$Q_v^+$}] at (000){};
\path (001v) -- (011v) node [above=10pt, midway, label=above:\Huge $\mathcal{Q}_v$]{};
\path (100) -- (110) node [below=10pt, midway, label=below left:\Huge $Q$]{};

\end{tikzpicture}
\caption{The $(d-2)$-diagram $\mathcal{Q}_v$ and an interior vertex $x = \pi_v(X^+) \cap \pi_v(X^-)$.}
\label{fig1}
\end{figure}
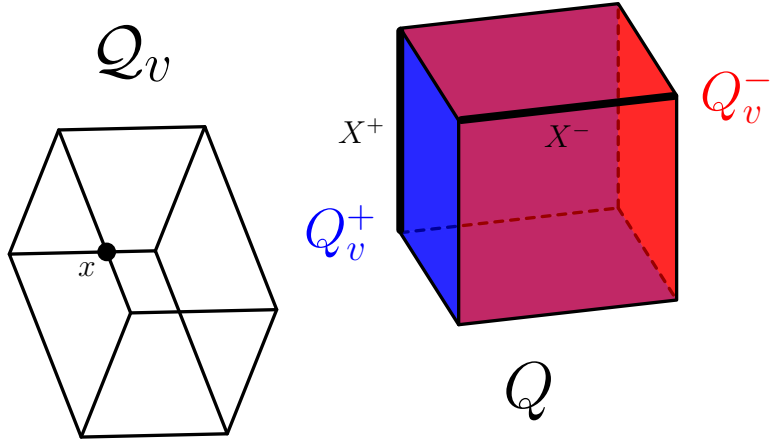

Let $\ell^+ = \dim X^+$ and $\ell^- = \dim X^-$. We can see that $\ell^+ + \ell^- \leq d-2$. Because $\pi_v(X^+)$ and $\pi_v(X^-)$ intersect nontrivially, we know $v$ is parallel to $\operatorname{aff}(X^+ \cup X^-)$, meaning $\operatorname{aff}(X^+ \cup X^-) = \mathds{R}^{d-1}$. Thus, $\ell^+ + \ell^- = d-2$. It follows that
\begin{gather*}
    \dim Q/X^+ = d - \ell^+ - 2 = \ell^-,\\
    \dim Q/X^- = d - \ell^- - 2 = \ell^+.
\end{gather*}
As a result,
\begin{gather*}
    f_{k - \ell^+ - 1}(Q/X^+) \geq {\ell^- + 1 \choose k - \ell^+} = {\ell^- + 1 \choose d - k - 1},\\
    f_{k - \ell^- - 1}(Q/X^-) \geq {\ell^+ + 1 \choose k - \ell^-} = {\ell^+ + 1 \choose d - k - 1}.
\end{gather*}
The above inequalities indicate that $Q$ has at least ${\ell^- + 1 \choose d - k - 1}$ $k$-dimensional faces containing $X^+$ and at least ${\ell^+ + 1 \choose d - k - 1}$ $k$-dimensional faces containing $X^-$. We have established that these two sets of faces are disjoint, so $Q$ has at least ${\ell^- + 1 \choose d - k - 1} + {\ell^+ + 1 \choose d - k - 1}$ $k$-dimensional faces not contained in $Q^+ \cap Q^-$. In other words,
\[
f_k(Q) - f_k(\pi_v(Q)) \geq {\ell^- + 1 \choose d - k - 1} + {\ell^+ + 1 \choose d - k - 1}.
\]
By Lemma \ref{convex}, it follows that for all $v \in \mathds{R}^{d-1}$ in general position with respect to $Q$,
\[
f_k(Q) - f_k(\pi_v(Q)) \geq {\lceil \frac{d}{2} \rceil \choose d-k-1} + {\lfloor \frac{d}{2} \rfloor \choose d-k-1} = 2\rho(d, d-k-1).
\]
Thus,
\[
f_k(Q) - \max_v f_k(\pi_v(Q)) \geq 2\rho(d, d-k-1).
\]
Finally, we can use Theorem \ref{perles} to conclude that
\[
\varphi_k(Q) \geq \frac{1}{2}\Bigl[f_k(P) - \max_v f_k(\pi_v(P))\Bigr] \geq \rho(d, d-k-1).\qedhere
\]
\end{proof}

Our main result follows.

\begin{theorem}
\label{bound}
Let $P$ be a convex $d$-polytope, and suppose $0 \leq k \leq d-1$. Then
\begin{align*}
\frac{f_k(P)}{f_0(P)} &\geq \rho(d, k),\\% \tag{\ref{eqn:bound}}\\
\frac{f_k(P)}{f_{d-1}(P)} &\geq \rho(d, d-k-1).% \tag{\ref{eqn:dual}}
\end{align*}
In the former, equality holds precisely when $k=0$ or when $k=1$ and $P$ is simple. In the latter, equality holds precisely when $k=d-1$ or when $k=d-2$ and $P$ is simplicial.
\end{theorem}

\begin{proof}
Let $P$ be a $d$-polytope and $0 \leq k \leq d-1$. Let $\mathcal{F}$ be the set of facets of $P$, and let $\mathcal{G}$ be the set of $k$-dimensional faces of $P$. For all $F \in \mathcal{F}$ and $G \in \mathcal{G}$, if $G$ is not a face of $F$, we set $\varphi(F, G) = 0$.

By Theorem \ref{curvature}, for all $G \in \mathcal{G}$, $\sum_{F \in \mathcal{F}} \varphi(F, G) \leq 1$. Equality is achieved if and only if $k \geq d-2$ (trivially if $k = d-1$). Thus,
\[
\sum_{\substack{F \in \mathcal{F} \\ {G \in \mathcal{G}}}} \varphi(F, G) \leq |\mathcal{G}| = f_k(P),
\]
with equality if and only if $k \geq d-2$.

By Proposition \ref{angles}, for all $F \in \mathcal{F}$, $\sum_{G \in \mathcal{G}} \varphi(F, G) \geq \rho(d, d-k-1)$. Thus,
\[
\sum_{\substack{F \in \mathcal{F} \\ {G \in \mathcal{G}}}} \varphi(F, G) \geq \rho(d, d-k-1)|\mathcal{F}| = \rho(d, d-k-1)f_{d-1}(P).
\]
It follows that
\[
\frac{f_k(P)}{f_{d-1}(P)} \geq \rho(d, d-k-1).
\]
This inequality is guaranteed to be strict if $k < d-2$.

It remains to determine when equality is achieved for $k \geq d-2$. We can observe that $\rho(d, 0) = 1$ and $\rho(d, 1) = \frac{1}{2}(\lceil \frac{d}{2} \rceil + \lfloor \frac{d}{2} \rfloor) = \frac{d}{2}$. Thus, equality is trivial for $k=d-1$. If $k=d-2$, then equality holds if and only if $\frac{f_{d-2}(P)}{f_{d-1}(P)} = \frac{d}{2}$; equivalently, if and only if each facet of $P$ contains exactly $d$ ridges of $P$. This occurs precisely when $P$ is simplicial.

By duality, for all $0 \leq k \leq d-1$,
\[
\frac{f_k(P)}{f_0(P)} \geq \rho(d, k).
\]
Equality holds precisely when $k=0$ or when $k=1$ and $P$ is simple.
\end{proof}

Observe that if $k \leq \lfloor \frac{d}{2} \rfloor$, then $\rho(d, k) \geq 1$. Likewise, if $k \geq \lceil \frac{d}{2} \rceil - 1$, then $\rho(d, d-k-1) \geq 1$. We can therefore answer B\'ar\'any's question:

\begin{corollary}
\label{barany}
For all $d$-polytopes $P$ and all $0 \leq k \leq d-1$,
\[
f_k(P) \geq \min\{f_0(P), f_{d-1}(P)\}.
\]
In particular, $f_k(P) \geq f_0(P)$ for $k \leq \lfloor \frac{d}{2} \rfloor$, and $f_k(P) \geq f_{d-1}(P)$ for $k \geq \lceil \frac{d}{2} \rceil - 1$.
\end{corollary}

%%%%%%%%%%%%%%%%%%%%%%%%%%%%%%%%%%%%%%%%%%%%%%%%%%%%%%%%%%%%%%%%%%%%%%%%%%%%%%%%
%%%%%%%%%%%%%%%%%%%%%%%%%%%%%%%%%%%%%%%%%%%%%%%%%%%%%%%%%%%%%%%%%%%%%%%%%%%%%%%%

\section{Acknowledgements}
The author would like to thank Isabella Novik for encouraging him to write this paper, as well as her incredible support and guidance throughout the writing process. The author would also like to thank Rowan Rowlands, Louis Billera, and Maria-Romina Ivan for their generous help in editing this paper.

\bibliography{bibliography}
\bibliographystyle{plain}

\end{document}